\documentclass[10pt]{amsart}
\usepackage{latexsym, amsmath,amssymb}

\renewcommand{\epsilon}{\varepsilon}
\newcommand{\newsection}[1]
{\subsection{#1}\setcounter{theorem}{0} \setcounter{equation}{0}
\par\noindent}

\newtheorem{theorem}{Theorem}

\newtheorem{thm}[theorem]{Theorem}
\newtheorem{cor}[theorem]{Corollary}
\newtheorem{lem}[theorem]{Lemma}
\newtheorem{prop}[theorem]{Proposition}
\theoremstyle{definition}
\newtheorem{defn}[theorem]{Definition}
\theoremstyle{remark}
\newtheorem{rem}[theorem]{Remark}
\newcommand{\norm}[1]{\left\Vert#1\right\Vert}

\newcommand{\cd}{\, \cdot\, }
\newcommand{\g}{{\mathrm g}}

\newcommand{\Real}{\mathbb R}

\newcommand{\less}{\lesssim}
\newcommand{\three}{\varepsilon}
\newcommand{\Rn}{{\mathbb R}^n}
\newcommand{\Z}{{\mathbb Z}}

\begin{document}

%{Title: General Strichartz Estimates and Applications}\\

%{Author: Xin Yu} \\

%{Address: Department of Mathematics, Johns Hopkins University, Baltimore, MD21212}\\

%\clearpage
\setcounter{page}{1}
\renewcommand{\thefootnote}{\fnsymbol{footnote}}
\footnote[0]{2000\textit{ Mathematics Subject Classification}.
 Primary 35L05, 35B20. Secondary 35B65, 35B33.}
\title[General Strichartz Estimates and Applications]
{Generalized Strichartz Estimates on Perturbed Wave Equation and Applications on Strauss Conjecture}
\author{Xin Yu\\
Department of Mathematics, Johns Hopkins University}
\email{xyu@math.jhu.edu}
\thanks{The author is sincerely grateful to Christopher Sogge for his patient guidance and many helpful suggestions during the study on this subject. She would also like to thank Chengbo Wang for some stimulating discussions and comments on Section 3.}
\begin{abstract}
 In this paper we prove a general Strichartz estimate for certain perturbed wave equation, and here we can drop the
 nontrapping hypothesis and handle trapping obstacles with some loss of derivatives for data in the local energy decay
 estimates. We then give the obstacle version of sharp life span for semilinear wave equations when $n=3,p<p_c$, by using
 the real interpolation method, and by getting corresponding finite time Strichartz estimates (see Section 3). Finally,
 as another application, we get the Strauss conjecture for semilinear wave
 equations with several convex obstacles when $n=3,4$ (see Section 4).
\end{abstract}

\maketitle
\newsection{Introduction and Main Result}

The purpose of this paper is to prove a general Strichartz estimate
for certain perturbed wave equation under known local energy decay
estimates, and as applications, we get the Strauss conjecture for
several convex obstacles in $n=3,4$. Our results improve on earlier
work in Hidano, Metcalfe, Smith, Sogge and Zhou \cite{KJHCY}.
Firstly, we can drop the nontrapping hypothesis and handle trapping
obstacles with some loss of derivatives for data in the local energy
decay estimates (see \eqref{hypo} below). The hypothesis
\eqref{hypo} is fulfilled in many cases in the nontrapping case
where there is local decay of energy with no loss of derivatives
(see \cite{V}, \cite{MelSj}, \cite{taylor}, \cite{burq},
\cite{MRS}). \eqref{hypo} is also known to hold in several examples
involving hyperbolic trapped rays ( see \cite{Ikawa1}, \cite{Ikawa},
\cite{Christianson}). Secondly, we give the obstacle version of
sharp life span for semilinear wave equations when $n=3$ and
$p<p_c$, by using the real interpolation between the KSS estimate
and the endpoint Trace Lemma, and by getting a corresponding finite
time Strichartz estimate. This blowup result complements to the
Strauss Conjecture stated in \cite{KJHCY}, which only dealt with the
global existence part($p>p_c$) (see Section 3). Lastly, we are able
to use the general Strichartz estimates we have gained to get the
Strauss conjecture for some perturbed semilinear wave equations with
trapped rays when $n=3,4$ (see Section 4).

We consider wave equations on an exterior domain $\Omega \subset
\Real^n$:
\begin{equation}
\label{obwaveequation} \left\{
\begin{array}{rll}
&(\partial_t^2-\Delta_g)u = F(t,x)\;{\rm on}\;\;\mathbb{R}_+\times \Omega  \\
&u|_{t=0}= f,  \;\partial_t u|_{t=0}= g,  \\
&(Bu)(t,x)=0,\;{\rm on}\;\;\mathbb{R}_+\times \partial\Omega,\\
\end{array}
\right.
\end{equation}
where for simplicity we take $B$ to be either the identity operator
or the inward pointing normal derivative $\partial_v$. The operator
$\Delta_g$ is the Laplace-Beltrami operator associated with a
smooth, time independent Riemannian metric $g_{jk}(x)$ which we
assume equals the Euclidean metric $\delta_{jk}$ for $|x|\geq R$,
some $R$. The set $\Omega$ is assumed to be either all of $\Real^n$,
or the complement a subset of $|x|<R$ with smooth boundary. Note
that here we do not require that $
\Rn\backslash\Omega$ is nontrapping.\\

We will make the following local decay assumption:\\
\noindent \textbf{Hypothesis B.} Fix the  boundary operator $B$ and
the exterior domain $\Omega\subset \Real^n$ as above. We then assume
that given $R_0>0$
\begin{multline}
\label{hypo}
\int_0^S (\norm{u(t,\cdot)}_{H^1(|x|<R_0)}^2 +
\norm{\partial_tu(t,\cdot)}_{L^2(|x|<R_0)}^2)\;dt \\
\less \norm f_{\dot{H}^{1+\three}}^2+\norm
g_{\dot{H}^\three}^2+\int_0^S \norm{F(s,\cdot)}_{\dot{H}^\three}^2\; ds,
\end{multline}
where $u$ is a solution of \eqref{obwaveequation} with data $(f,g)$
and forcing term $F$ that both vanish for $|x|>R_0$.\\

\begin{rem}
\label{rem1.1} We assume $S$ to be finite time $T$ or $\infty$.
Moreover, although $\three$ could be any real number without
effecting our techniques much, here we assume $\three\geq 0$ is an
arbitrarily small number (which is all we need for now) throughout
the paper for clarity of explanation. Note that when $\three=0$ and
$T=\infty$, it is just the case in \cite{KJHCY}. More specifically,
when the obstacle is nontrapping, and $\Delta$ is  the standard
Euclidean Laplacian, we will have that local energy decays
exponentially in odd dimensions $n\geq 3$ and polynomially in even
dimensions except $n=2$ (\cite{Mel}); For $n=2$, local energy decays
like $O((\log(2+t))^{-2}(1+t)^{-1})$ (\cite{V}). These results imply
\eqref{hypo}. When $\Delta_g$ is a time-independent variable
coefficient compact perturbation of $\Delta$, one also has that
\eqref{hypo} is valid for the Dirichlet-wave equation for $n\geq 3$
as well for $n=2$ if $\partial\Omega\neq \emptyset$
(\cite{taylor},\cite{burq}). On the other hand, when there are
trapped rays, it is known that a uniform decay rate is generally not
possible (\cite{Ral1}), but we can get some local energy decay by
trading some derivatives in the initial data. Ikawa, for example,
got the following exponential decay when $n=3$ and there are several
convex obstacles which are far apart (see \cite{Ikawa}),
\begin{equation*}
\norm{u'(t,x)}_{L_x^2(|x|<1)}\less e^{-at}\norm {u'(0,x)}_{\dot H^{2}(|x|<1)},\quad {\rm where}\ a \;{\rm is \ a \ constant.}\\
\end{equation*}
 By interpolation between this estimate and the standard energy estimates, it is easy to get
\begin{equation}
\label{ikawaest}\norm{u'(t,x)}_{L_x^2(|x|<1)}\less e^{-ct}\left (\norm f_{\dot H^{1+\three}(|x|<1)}+\norm g_{\dot H^{\three}(|x|<1)}\right ),\quad {\rm where}\ c \;{\rm is \ a \ constant,}\\
\end{equation}
which implies our Hypothesis B. When there is only one hyperbolic
trapped ray, Christianson (\cite{Christianson}) also showed that for
all odd dimensions $n\geq3$ we have the local energy decay
\begin{equation}
\norm{u'(t,x)}_{L_x^2(|x|<1)}\less e^{-t^{1/2}/C}\norm
{u'(0,x)}_{\dot H^{1+\three}(|x|<1)},
\end{equation}
which gives Hypothesis B as well. Further work in this direction can
be seen in \cite{Burq1}, \cite{Burq2}, \cite{Christianson1},
\cite{Datchev}.
\end{rem}

In order to deal with the extra derivatives in \eqref{hypo}, we
introduce a Sobolev-type norm as follows.
\begin{defn}
Define $\tilde{H}^\gamma_\three(\Real^n)$ (and
$\tilde{H}^\gamma_\three(\Omega)$) to be the space with norm
\begin{equation}
\label{norm1}
\norm h _{\tilde{H}^\gamma_\three(\Rn)}=\norm {|D|^\gamma(1-\Delta)^{\frac\three 2}h}_{L^2_x(\Rn)}=\left (\int_{\Real^n}\big ||\xi|^\gamma(1+|\xi^2|)^{\frac{\three}{2}}\hat h(\xi)\big |^2\;d\xi\right )^{1/2}.\\
\end{equation}
\end{defn}
Notice that if $0\leq\three_1<\three_2$, then
$$\tilde{H}^\gamma_{\three_2}\subset \tilde{H}^\gamma_{\three_1}.$$
We also notice that, when $\three\geq0$ the above norm is equivalent
to the following useful form:
\begin{equation}
\label{norm2} \norm h _{\tilde{H}^\gamma_\three(\Rn)}\approx \norm
h_{\dot{H}^\gamma(|\xi|<1)}+\norm{h}_{\dot{H}^{\gamma+\three}(|\xi|>1)}\approx
\norm h_{\dot{H}^\gamma(\Rn)}+\norm{h}_{\dot{H}^{\gamma+\three}(\Rn)} .\\
\end{equation}

\begin{rem}
Similarly we can define the norm on a manifold $\Omega$ as in
\cite{KJHCY} and \cite{SS}. Roughly speaking,
\begin{equation*}
\|f\|_{\tilde H_\three^\gamma(\Omega)}=\|\beta f\|_{\tilde
H_\three^\gamma(\Omega')} +\|(1-\beta)f\|_{\tilde
H_\three^\gamma(\Real^n)},
\end{equation*}
where $\beta\in \mathbb{C}_0^\infty$ is supported in $|x|<2R$ and
equals $1$ in $|x|<R$, and $\Omega'$ is the embedding of
$\Omega\cap\{|x|<2R\}$ into the torus obtained by periodic extension
of $\Omega\cap [-2R,2R]^n$, so that $\partial\Omega' =
\partial\Omega$. Here the spaces $\tilde H_\three^\gamma(\Omega')$
are defined by a spectral decomposition of $\Delta_\g|_{\Omega'}$
subject to the boundary condition $B$.

When $\three=0$, it is easy to see our definitions coincides with
the homogenous Sobolev spaces and we have
$$\tilde H_0^\gamma(\Rn)=\dot H^\gamma(\Rn), ~\tilde H_0^\gamma(\Omega)=\dot H^\gamma(\Omega).$$
\end{rem}
Now we can redefine `admissible` using the above Sobolev-type norm.
\begin{defn}
We say that $(X,\gamma,\eta, p)$ is almost admissible if
it satisfies\\
i), Minkowski almost Strichartz estimates
\begin{equation}
\label{1.11} \norm u_{L_t^pX([0,S]\times\Real^n)}\less A(S)\left
(\norm{u(0,\cdot)}_{\dot{H}^\gamma(\Real^n)}+\norm{\partial_tu(0,\cdot)}_{\dot{H}^{\gamma-1}(\Real^n)}\right
),
\end{equation}
where $A(S)$ is a function of $S$ and equals a constant when
$S=\infty$.

\noindent ii), Local almost Strichartz estimates for $\Omega$
\begin{equation}
\label{1.12} \norm u_{L_t^pX([0,1]\times\Omega)}\less
\norm{u(0,\cdot)}_{\tilde{H}^\gamma_\eta(\Omega)}+\norm{\partial_tu(0,\cdot)}_{\tilde{H}^{\gamma-1}_\eta(\Omega)}.
\end{equation}
\end{defn}

Notice that the factor $A(S)$ only depends on the time $S$, and we
only consider when the time $S$ is large and $A(S)\gtrsim 1$. Also
notice that here we assume a weaker local Strichartz estimates by
losing some derivatives in the regularity of initial data, which
probably will happen when there are broken rays in the manifold. We
also assume $\eta\geq0$ is an arbitrarily small number in our
theorems, and actually in our applications we only need the case
where $\eta=0$.

We will assume $1-({n}/2)<\gamma<{n}/2$ throughout, so that
$(H_\gamma,\dot{H}_\gamma)$ and $(H_{1-\gamma},\dot{H}_{1-\gamma})$
are comparable pairs for functions supported in a ball. Besides, for
$\beta\in C_0^\infty(\Rn)$, with $\beta=1$ on a neighborhood of
$\Rn\backslash\Omega$, we assume that
$$
\norm{(1-\beta)f}_{X(\Omega)}\approx\norm{(1-\beta)f}_{X(\Rn)}.$$

Now we will state our main Strichartz estimates:
\begin{thm}
\label{mainest2} Let $n>2$ and assume that $(X,\gamma, \eta,p)$ is
almost admissible with
\begin{equation}
\label{pcondition} p>2 \;{\rm and}\; \gamma\in
[-\frac{n-3}{2},\frac{n-1}{2}).
\end{equation}
Then if Hypothesis B is valid and if $u$ solves
\eqref{obwaveequation} with forcing term $F=0$, we have the abstract
Strichartz estimates
\begin{equation}
\label{se} \norm u_{L_t^pX([0,S]\times\Omega)}\less
A(S)(\norm{f}_{\tilde{H}^\gamma_{\three+\eta}(\Omega)}+\norm{g}_{\tilde{H}^{\gamma-1}_{\three+\eta}(\Omega)}).
\end{equation}
\end{thm}

\begin{rem} We need $-(n-3)/2\leq\gamma<({n-1})/2$ since we will use
Lemma \ref{2.2}, which requires $\gamma+\three+\eta\leq\frac{n-1}2$
for a sufficiently small number $\three$, thus precisely $\gamma\in
[-(n-3)/{2}, ({n-1})/{2}-\three-\eta]$. On the other hand, when
$\three$ is allowed to take large values, which depends on our local
energy estimates, we can easily adapt our arguments to show
$$\norm u_{L_t^pX([0,S]\times\Omega)}\less
A(S)(\norm{f}_{\tilde{H}^\gamma_{2\three+\eta}(\Omega)}+\norm{g}_{\tilde{H}^{\gamma-1}_{2\three+\eta}(\Omega)})$$
under the assumption $\gamma\in[-(n-3)/2,(n-1)/2]$.
\end{rem}

Next we will see two corollaries that involve adding forcing term to
the equation.
\begin{cor}
\label{cor1.5}  Assume that $(X,\gamma,\eta,p)$ and
$(Y,1-\gamma,\eta,r)$ are almost admissible  and that Hypothesis B
is valid. Also assume that \eqref{pcondition} holds for
$(X,\gamma,\eta,p)$ and $(Y,1-\gamma,\eta,r)$. Then we have the
following global abstract Strichartz estimates for the solution of
\eqref{obwaveequation}
\begin{equation}\label{mainest4}
\|u\|_{L^p_tX([0,S]\times\Omega)}\lesssim A(S)(\|f\|_{\tilde
H_{\three+\eta}^{\gamma}(\Omega)} +\|g\|_{\tilde
H^{\gamma-1}_{\three+\eta}(\Omega)})
+A^2(S)\|\Lambda^{2(\three+\eta)}F\|_{L^{r'}_tY'([0,S]\times
\Omega)},
\end{equation}
where $\Lambda=(1-\Delta)^{1/2}$, $r'$ denotes the conjugate
exponent to $r$ and $\|\, \cd\|_{Y'}$ is the dual norm to
$\|\cd\|_Y$.
\end{cor}

\begin{proof}
Since
$$\norm h _{\tilde{H}^\gamma_\three} = \norm{|D|^\gamma\Lambda^\three h}_{L^2},\quad \three\geq0,$$
it is easy to see that the dual norm is
$$\norm h _{(\tilde{H}^\gamma_\three)'}=\norm h _{\tilde{H}^{-\gamma}_{-\three}}=\norm{|D|^{-\gamma}\Lambda^{-\three} h}_{L^2}.$$
To prove \eqref{mainest4}, we may assume by \eqref{se} that the
initial data vanishes. If $|D|=\sqrt{-\Delta_\g}$ is the square root
of minus the Laplacian (with the boundary conditions $B$), then we
need to show
\begin{multline}
\Bigl\|\,\int_0^t e^{-i(t-s)|D|}|D|^{-1}F(s,\cd)\, ds\,
\Bigr\|_{L_t^pX([0,S]\times\Omega)}\\
\lesssim A^2(S)\|\Lambda^{2(\three+\eta)}F\|_{L^{r'}_tY'([0,S]\times \Omega)}\,.
\end{multline}

We have $p>2>r'$, So \eqref{se} and an application of the
Christ-Kiselev Lemma (cf. \cite{CK}) imply that it suffices to prove
the estimate
\begin{multline}
\label{to}
\Bigl\|\,\int_0^S e^{-is|D|}|D|^{-1}F(s,\cd)\, ds\,
\Bigr\|_{\tilde H^\gamma_{\three+\eta}(\Omega)}=\Bigl\|\,\int_0^S e^{-is|D|}|D|^{-1+\gamma}\Lambda^{\three+\eta}F(s,\cd)\, ds\,
\Bigr\|_{L^2(\Omega)}\\
\lesssim A(S)\|\Lambda^{2(\three+\eta)}F\|_{L^{r'}_tY'([0,S]\times \Omega)}\,.
\end{multline}
Note that duality of \eqref{se} for $(Y,1-\gamma,\eta,r)$ gives
$$\Bigl\|\,\int_0^S e^{-is|D|}F(s,\cd)\, ds\,
\Bigr\|_{\tilde{H}_{-\three-\eta}^{\gamma-1}(\Omega)} \less
A(S)\|F\|_{L^{r'}_tY'([0,S]\times \Omega)}\,,
$$
i.e.
\begin{equation}
\label{from}
\Bigl\|\,\int_0^S e^{-is|D|}|D|^{\gamma-1}\Lambda^{-\three-\eta} F(s,\cd)\, ds\,
\Bigr\|_{L^2(\Omega)}
\less A(S)\|F\|_{L^{r'}_tY'(\Real_+\times \Omega)}\,.
\end{equation}
Now \eqref{to} follows from \eqref{from}.\\
\end{proof}

As a special case of \eqref{mainest4} when the spaces $X$ and $Y$
are the standard Lebesgue spaces, we have the following obstacle
version of Strichartz estimates in Minkowski space obtained by
Keel-Tao \cite{KT}.

\begin{cor}\label{mixed}  Suppose that $n\ge 3$ and that Hypothesis B is
valid.  Suppose that $p,r>2$, $q,s\ge 2$ and that
$$
\frac1p+\frac{n}q=\frac{n}2-\gamma=\frac1{r'}+\frac{n}{s'}-2
$$
and
$$
\frac2p+\frac{n-1}q\,, \, \frac2{r}+\frac{n-1}{s}\le \frac{n-1}2 .
$$
Then if the local Strichartz estimate \eqref{1.12} holds
respectively for $\bigl(L^q(\Omega),\gamma,\eta,p\bigr)$ and
$\bigl(L^{s}(\Omega),1-\gamma,\eta,r\bigr)$, it follows that when
$u$ solves \eqref{obwaveequation},
$$
\|u\|_{L^p_tL^q_x(\Real_+\times\Omega)}
\lesssim \|f\|_{\tilde H_{\three+\eta}^\gamma(\Omega)}+\|g\|_{\tilde H^{\gamma-1}_{\three+\eta}(\Omega)}+
\|(1-\Delta)^{\three+\eta}F\|_{L^{r'}_t\!L^{s'}_x(\Real_+\times\Omega)}.
$$
\end{cor}

Note that the estimates we obtain here do not cover the endpoint
case $q=2$ or $\tilde q=2$, or the case $n=2$, which are proved to
be true for $\Rn$ and $\Delta_g=\Delta$ by Keel-Tao \cite{KT}. As
for the assumption \eqref{1.12}, Smith and Sogge \cite{SS0} showed
that it holds when $\Omega$ is the exterior of a geodesically convex
obstacle, and their results apply to the case where there are
finitely many convex obstacles by finite propagation of speed. More
work in this direction can be found in \cite{BLP}, \cite{BP},
\cite{BSS}, \cite{SS2}, but only partial results with smaller range
of $q,~r,~\gamma$ have been gained.

%%%%%%%%%%%%%%%%%%%%%%%%%%%%%%%%%%%%%%%%%%%%%%%%%%%%%%%%%%%%%%%%%%%%%%%%%%%%%%%%%%%%%%%%%%%%%%%%%%%%%%%%%%%%%%%%%%%%%%%

\newsection{Proof of Theorem ~\ref{mainest2}}

In this section we will see how local Strichartz estimates and
global Strichartz estimates in Minkowski space imply the (almost)
global Strichartz estimates in a general domain $\Omega$. The first
Lemma is key to achieve Proposition \ref{prop2.11} and will be used
in Theorem \ref{mainest2} as well.
\begin{lem}
\label{2.2}
Fix $\beta\in C_0^\infty(\Real^n)$ and assume that
$\gamma\leq \frac{n-1}{2}$. Then
\begin{equation}
\label{free est1} \int_{-\infty}^\infty \norm
{\beta(\cdot)(e^{it|D|}f)(t,\cdot)}_{H^\gamma(\Real^n)}^2\; dt \less
\norm f_{\dot{H}^\gamma(\Real^n)}^2,
\end{equation}
if $|D|=\sqrt{-\Delta}$.
\end{lem}

\begin{proof}
Refer to Lemma 2.2 in \cite{SS}.
\end{proof}

Now we introduce a result which will be used to control the solution
of \eqref{obwaveequation} away from the obstacle.
\begin{prop}
\label{prop2.11} Consider the wave equation
\begin{equation}
\label{inhomowaveequation} \left\{
\begin{array}{rll}
(\partial_t^2-\Delta)u &=& F(t,x)\quad {\rm on}\;\;\mathbb{R}_+\times \mathbb{R}^n  \\
u|_{t=0}&=& f,  \\
\partial_t u|_{t=0}&=& g.
\end{array}
\right.
\end{equation}
Let $w$ be a solution with $f=g=0$, and assume that \eqref{1.11} is
valid whenever $v$ is a solution of the homogeneous wave equation.
Assume further that $p>2, \gamma\geq({n-3})/2$. Then, if
$$F(t,x)=0 \quad {\rm if} \; |x|>2R,$$
we have
\begin{equation}
\norm w_{L_t^pX([0,S]\times\Real^n)}\less A(S)\norm
F_{L_t^2\dot{H}^{\gamma-1}([0,S]\times\Real^n)}.
\end{equation}
\end{prop}

\begin{proof}
When $S=\infty$, this is just Proposition 2.1 in \cite{KJHCY}.
Moreover, their argument is easily modified to give the proof when
$S=T$ is finite.
\end{proof}

The next Lemma gives two useful local decay estimates. The first one
is necessary to prove Theorem \ref{mainest2}, and the second one
will be applied in the next two sections.
\begin{lem}
Let $u$ solve \eqref{obwaveequation} and assume that Hypothesis B
holds. Let $\beta\in C_0^\infty(\Real^n)$ equal 1 on a neighborhood
of $\Real^n\backslash\Omega$.
Then we have the following estimates:\\
i), if $f,g,$ and $F$ are supported in $|x|<2R$,
\begin{equation}
\begin{aligned}
\label{energyest1}
\norm {\beta u}_{L_t^\infty H_B^\gamma([0,S]\times\Omega)}+\norm {\beta \partial_t u}_{L_t^\infty H_B^{\gamma-1}([0,S]\times\Omega)}&+\norm{\beta u}_{L_t^2H_B^\gamma([0,S]\times\Omega)}+\norm{\beta \partial_tu}_{L_t^2H_B^{\gamma-1}([0,S]\times\Omega)}\\
&\hspace{-0.3 in}\less \norm f_{\dot{H}_B^{\gamma+\three}(\Omega)}+
\norm g_{\dot{H}_B^{\gamma+\three-1}(\Omega)}+ \norm
F_{L_t^2\dot{H}_B^{\gamma+\three-1}([0,S]\times\Omega)}.
\end{aligned}\\
\end{equation}
ii), if F is supported in $|x|<R$, $\gamma<\frac{n-1}2$,
\begin{equation}
\begin{aligned}
\label{energyest2}
\norm {u}_{L_t^\infty \dot{H}_B^\gamma([0,S]\times\Omega)}+\norm {\partial_t u}_{L_t^\infty \dot{H}_B^{\gamma-1}([0,S]\times\Omega)}&+\norm{\beta u}_{L_t^2H_B^\gamma([0,S]\times\Omega)}+\norm{\beta \partial_tu}_{L_t^2H_B^{\gamma-1}([0,S]\times\Omega)}\\
&\hspace{-0.3 in}\less \norm f_{\tilde{H}_\three^{\gamma}(\Omega)}+
\norm g_{\tilde{H}_\three^{\gamma-1}(\Omega)}+ \norm
F_{L_t^2\dot{H}_B^{\gamma+\three-1}([0,S]\times\Omega)}.
\end{aligned}
\end{equation}\\
\end{lem}

\begin{proof}
First note that the space $H^{\gamma}_B(\Omega)$ is the usual
Dirichlet space with compatibility conditions on boundary of
$\Omega$ satisfied, and we are assuming that the required
compatibility conditions on data are meet throughout the paper(see
for example in \cite{SS}), therefore write $H^{\gamma}(\Omega)$ for
short elsewhere.

i), Note that $f,g,$ and $F$ are supported in a ball, the $L^2_t$
estimate in the case $\gamma=1$ is just \eqref{hypo}, and then by
elliptic regularity of the operator $\Delta_g$,
\begin{eqnarray}\label{1} \norm{\beta u}_{L_t^2H_x^3}+\norm{\beta
\partial_tu}_{L_t^2H_x^{2}}&\less&
\norm{\Delta_g(\beta u)}_{L_t^2H_x^1}+\norm{\beta
u}_{L_t^2H_x^{1}}+\norm{\Delta_g(\beta
\partial_t u)}_{L_t^2L^2_x}+\norm{\beta\partial_t
u}_{L_t^2L^2_x}\nonumber\\
&\less& \norm{\beta\Delta_g u}_{L_t^2H_x^1}+\norm{[\Delta_g,\beta]
u}_{L_t^2H_x^1}+\norm{\beta u}_{L_t^2H_x^{1}}\nonumber\\
&&+\norm{\beta\partial_t\Delta_g
u}_{L_t^2L^2_x}+\norm{[\Delta_g,\beta]
\partial_t u}_{L_t^2L^2_x}+\norm{\beta\partial_t
u}_{L_t^2L^2_x}
\end{eqnarray}
Since $\Delta_g u$ solves the equation with data $(\Delta_g f,
\Delta_g g)$ and forcing term $\Delta_g F$, we get
\begin{eqnarray}
\label{2} \norm{\beta\Delta_g
u}_{L_t^2H_x^1}+\norm{\beta\partial_t\Delta_g
u}_{L_t^2L^2_x}&\less&\norm{\Delta_g f}_{\dot
H^{1+\three}_x}+\norm{\Delta_g g}_{\dot H^{\three}_x}+\norm{\Delta_g
F}_{L^2_t\dot H^{\three}_x}\nonumber\\
&\less&\norm{f}_{\dot H^{3+\three}_x}+\norm{ g}_{\dot
H^{2+\three}_x}+\norm{F}_{L^2_t\dot H^{2+\three}_x}.
\end{eqnarray}

Also, notice that $[\Delta_g,\beta]u=\beta_1\partial_x u+\beta_2 u$,
where $\beta_i\in C_0^\infty,\; i=1,2$ have support belonging to
$\text{supp} (\beta)$. Thus
\begin{eqnarray}
\label{3} \norm{[\Delta_g,\beta]
u}_{L_t^2H^1_x}+\norm{[\Delta_g,\beta]
\partial_t u}_{L_t^2L^2_x}&\less&\norm{\beta_3
u}_{L_t^2H^2_x}+\norm{\beta_3
\partial_t u}_{L_t^2H^1_x}\nonumber\\
&\less&\|\beta_3u\|_{L_t^2H^1}^\theta\|\beta_3u\|_{L_t^2H^3}^{1-\theta}+\norm{\beta_3
\partial_t u}_{L_t^2L^2_x}^{\theta}\norm{\beta_3
\partial_t u}_{L_t^2H^2_x}^{1-\theta},
\end{eqnarray}
where $\beta_3\in C_0^\infty$ has support in $\text{supp}
(\beta_1)\cap \text{supp} (\beta_2)$, and $\theta$ is any real
number in $(0,1)$.

Based on \eqref{1}, \eqref{2} and \eqref{3}, we get that the $L^2_t$
estimate is true for $\gamma=3$, and similarly holds for
$\gamma=5,7,9..$. and moreover for $\gamma\in \Real$ by duality and
interpolation, i.e.
\begin{multline}\label{L_2^t}
\norm{\beta u}_{L_t^2H_B^\gamma([0,S]\times\Omega)}+\norm{\beta \partial_tu}_{L_t^2H_B^{\gamma-1}([0,S]\times\Omega)}\\
\less \norm f_{\dot{H}^{\gamma+\three}(\Omega)}+ \norm
g_{\dot{H}^{\gamma+\three-1}(\Omega)}+ \norm
F_{L_t^2\dot{H}_B^{\gamma+\three-1}([0,S]\times\Omega)}.
\end{multline}
By Duhamel's principle, the inhomogeneous solution $v$ satisfies
$$\norm{\beta v}_{L_t^2H_B^\gamma([0,S]\times\Omega)}+\norm{\beta \partial_tv}_{L_t^2H_B^{\gamma-1}([0,S]\times\Omega)}\less \norm
F_{L_t^1\dot{H}_B^{\gamma+\three-1}([0,S]\times\Omega)},
$$
by duality of the above estimate, energy estimates and elliptic
regularity, we get
\begin{multline}
\label{infty} \norm {u}_{L_t^\infty
\dot{H}_B^\gamma([0,S]\times\Omega)}+\norm {\partial_t
u}_{L_t^\infty \dot{H}_B^{\gamma-1}([0,S]\times\Omega)}
\\\less \norm f_{\dot{H}^{\gamma}(\Omega)}+ \norm
g_{\dot{H}^{\gamma-1}(\Omega)}+ \norm
F_{L_t^2\dot{H}_B^{\gamma+\three-1}([0,S]\times\Omega)}.
\end{multline}
Now \eqref{energyest1} is a result of \eqref{L_2^t} and
\eqref{infty}.

ii), We first handle with the $L_t^2$ bounds. For the homogeneous
solution $v$, we can assume $f=g=0$ for $|x|\leq 3R/2$ by i).
Decompose $v=(1-\eta)v_0+\tilde{v}$, where $\eta\in C_0^\infty(\Rn)$
equals 1 for $|x|<R$ and vanishes for $|x|>3R/2$, and $v_0$ solves
the homogeneous wave equation in $\Rn\times\Real$. It is easy to see
$(1-\eta)v_0$ solves the Cauchy problem for the Minkowski space wave
equation with initial data $((1-\eta)f,(1-\eta)g)$ and forcing term
$G=\Delta\eta v_0+2\nabla\eta\cdot\nabla v_0$; $\tilde{v}$ solves
the wave equation with initial data $(0,0)$ and forcing term $-G$.
Since $G$ is supported in $R<|x|<2R$, we get the $L_t^2$ bounds for
$(1-\eta)v_0$ and $\tilde v$ by Lemma \ref{2.2} and i).

We get the $L^2_t$ bounds for the inhomogeneous solution $w$ follow
from i) since $F$ is still compactly supported.

Similarly to i), the $L_t^\infty$ bounds for $u=v+w$ follow also
from energy estimates , elliptic regularity and duality.
\end{proof}

The last proposition is a result of Proposition \ref{prop2.11} and
\eqref{energyest1}.
\begin{prop}
\label{2.33} Let $u$ solve \eqref{obwaveequation} and assume that
\begin{equation}
\label{dataspt} f(x)=g(x)=F(t,x)=0, \quad {\rm when}\; |x|>2R.
\end{equation}
If $(X,\gamma,\eta,p)$ is almost admissible with $p>2,
\gamma\geq-\frac{n-3}2$, and Hypothesis B holds, then we have
\begin{equation}
\label{2.5} \norm u_{L_t^pX([0,S]\times\Omega)}\less A(S)(\norm
f_{\dot{H}^{\gamma+\three+\eta}}+\norm
g_{\dot{H}^{\gamma+\three+\eta-1}}+\norm
F_{L_t^2\dot{H}^{\gamma+\three+\eta-1}}).
\end{equation}
\end{prop}

\begin{proof}
Fix $\beta\in C^\infty_0(\Rn)$ satisfying $\beta(x)=1$, $|x|\le 3R$
and write
$$u=v+w, \quad \text{where } \, v=\beta u, \, \, w=(1-\beta)u.$$
Then $w$ solves the free wave equation
\begin{equation*}
\begin{cases}
(\partial_t^2-\Delta)w=[\beta,\Delta]u
\\
w|_{t=0}=\partial_tw|_{t=0}=0.
\end{cases}
\end{equation*}
Notice that $[\beta,\Delta]u$ is compactly supported, so an
application of Proposition \ref{prop2.11} shows that
$\|w\|_{L^p_tX}$ is dominated by $A(S)\|\rho u\|_{L^2_t\dot
H^{\gamma}_B}$ if $\rho\in C_0^\infty$ equals one on the support of
$\beta$. Therefore, by \eqref{energyest1}, $\|w\|_{L^p_tX}$ is
dominated by the right hand side of \eqref{2.5} with $\eta=0$.

For $v=\beta u$, we decompose it in time $t$ and write
$v=\sum_{j=-\infty}^\infty \varphi(t-j)v$, where $\varphi \in
C^\infty_0((-1,1))$. Let $v_j=\varphi(t-j)v$ for $j\geq 1$ and
$v_0=v-\sum_{j=1}^\infty v_j$. Then $v_j$ solves
\begin{equation*}
\begin{cases}
(\partial_t^2-\Delta_\g)v_j =G_j
\\
Bv_j(t,x)=0,\quad x\in \partial\Omega
\\
v_j(0,\cd)=\partial_tv_j(0,\cd)=0,
\end{cases}
\end{equation*}
where $G_j=-\varphi(t-j)[\Delta_g,\beta]u +
[\partial^2_t,\varphi(t-j)]\beta u +\varphi(t-j)F$. Also $v_0$
solves the equation with $G_0=-\tilde \varphi [\Delta_g,\beta]u +
[\partial_t^2,\tilde \varphi]\beta u + \tilde \varphi F$ and initial
date $(f,g)$.

Since $G_j$ with $j\geq 0$ vanishes if $t$ is not in $[j-1,j+1]$ or
if $|x|>3R$, by the local Strichartz estimates \eqref{1.12} and
Duhamel, we get for $j=1,2,\dots$,
$$\|v_j\|_{L^p_tX([0,S]\times \Omega)}\less
\int_0^S \|G_j(s,\cd)\|_{\tilde{H}^{\gamma-1}_\eta}\, ds \lesssim
\|G_j\|_{L^2_t H^{\gamma+\eta-1}_B}.$$ Similarly,
$$
\|v_0\|_{L^p_tX(\Real_+\times \Omega)}\lesssim
\|f\|_{\tilde{H}^\gamma_\eta}+\|g\|_{\tilde{H}^{\gamma-1}_\eta}+\|G_0\|_{L^2_tH^{\gamma+\eta-1}_B}\,.
$$
Since $p>2$, by \eqref{energyest1} and the disjoint support of
$G_j$, we have
$$
\begin{aligned}
\|v\|^2_{L^p_tX([0,S]\times \Omega)}&\lesssim \sum_{j=0}^\infty
\|v_j\|^2_{L^p_tX([0,S]\times \Omega)}\\
&\less \sum_{j=1}^\infty \|G_j\|^2_{L^2_tH^{\gamma+\eta-1}_B([0,S]\times \Omega)}+\|v_0\|_{L^p_tX(\Real_+\times \Omega)}\\
&\lesssim \norm f^2_{\dot{H}^{\gamma+\three+\eta}}+\norm
g^2_{\dot{H}^{\gamma+\three+\eta-1}}+\norm
F^2_{L_t^2\dot{H}^{\gamma+\three+\eta-1}}, \\
&\less A^2(S)(\norm f^2_{\dot{H}^{\gamma+\three+\eta}}+\norm
g^2_{\dot{H}^{\gamma+\three+\eta-1}}+\norm
F^2_{L_t^2\dot{H}^{\gamma+\three+\eta-1}}),
\end{aligned}
$$
which finishes the proof of Proposition \ref{2.33}.
\\
\end{proof}

\noindent{\bf Proof of Theorem \ref{mainest2}:} By Proposition
\ref{2.33} we can assume that the initial data for $u$ vanishes when
$|x|<3R/2$.  Then we use a cutoff function $\beta\in
C^\infty_0(\Rn)$ satisfying $\beta(x)=1$, $|x|\le R$, and
$\beta(x)=0$, $|x|>3R/2$, and write
$$
u=u_0-v = (1-\beta)u_0 + ( \beta u_0-v)\,,
$$
where $u_0$ solves the Cauchy problem for the Minkowski space wave
equation. By the free estimate \eqref{1.11}, Proposition
\ref{prop2.11},  Lemma \ref{2.2} and energy estimates,
\begin{equation}
\label{4}
\begin{aligned}
\norm{(1-\beta)u_0}_{L_t^pX([0,S]\times\Rn)}&\less
A(S)(\|f\|_{\dot{H}^\gamma}+\|g\|_{\dot{H}^{\gamma-1}}+\|G\|_{L^2_t\dot{H}^{\gamma-1}_B})\\
&\less A(S)(\|f\|_{\dot{H}^\gamma}+\|g\|_{\dot{H}^{\gamma-1}}),
\end{aligned}
\end{equation}
where $G=\Box_g((1-\beta)u_0)=\Box( (1-\beta)u_0)=\Delta\beta\cd
u_0+2\nabla\beta\cd \nabla u_0$ is supported in $R<|x|<3R/2$.

Now consider $\tilde u= \beta u_0-v$, which has forcing term $-G$
and zero initial data. Again by Proposition \ref{2.33} and Lemma
\ref{2.2},
\begin{equation}
\label{5}
\begin{aligned}
\norm{\beta u_0-v}_{L_t^pX([0,S]\times\Omega)}&\less A(S)\|G\|_{L^2_t\dot{H}^{\gamma+\three+\eta-1}_B}\\
&\less A(S)\norm{\rho u_0}_{L_t^2H^{\gamma+\three+\eta}}\quad (\text{here\ } \rho\ \text{is a } C_0^\infty \ \text{function}.)\\
&\less
A(S)(\|f\|_{\dot{H}^{\gamma+\three+\eta}}+\|g\|_{\dot{H}^{\gamma+\three+\eta-1}}).
\end{aligned}
\end{equation}
Based on \eqref{4} and \eqref{5}, we have the theorem proved.\\
\qed

%%%%%%%%%%%%%%%%%%%%%%%%%%%%%%%%%%%%%%%%%%%%%%%%%%%%%%%%%%%%%%%%%%%%%%%%%%%%%%%%%%%%%%%%%%%%%%%%%%%%%%%%%%%%%%%%%%%%%%%

\newsection{Application 1: Sharp life span bounds for $p<p_c$ when $n=3$}

First let us describe the wave equation we will consider:
\begin{equation}
\label{eqn1} \left\{
\begin{array}{rll}
&(\partial_t^2-\Delta_g)u = F_p(u(t,x))\;{\rm on}\;\;\mathbb{R}_+\times \Omega  \\
&u|_{t=0}= f, \partial_t u|_{t=0}= g,  \\
&(Bu)(t,x)=0,\;{\rm on}\;\;\mathbb{R}_+\times \partial\Omega,\\
\end{array}
\right.
\end{equation}
with $B$ as above, the set $\Omega$ is assumed to be either all of
$\Real^3$, or else $\Omega=\Real^3\backslash \kappa$ where $\kappa$
is a compact subset of $|x|<R$ with smooth boundary. Also we assume
$\kappa$ is nontrapping in the sense that any geodesic restricted to
$|x|<R$ has bounded length.\\
We will assume that the nonlinear term behaves like $|u|^p$ when $u$
is small, and so we assume that
\begin{equation}
\sum_{0\leq j\leq 2} |u|^j|\partial_u^j F_p(u)| \less |u|^p,
\end{equation}
when u is small.

On the basis of the discussion in the first section (Remark
\ref{rem1.1}), we will assume Hypothesis B holds with $\three=0,
S=T$. Now if we set
$$\{Z\}=\{\partial_l, x_j\partial_k-x_k\partial_j: 1\leq l\leq n, 1\leq j<k\leq n\},$$
then we  have the following existence theorem for (\ref{eqn1}).

\begin{thm}
\label{existthm1} Let $n=3$, and fix $\Omega\subset\Real^n$ and the
boundary operator $B$ as above. Assume further that Hypothesis B is
valid with $\three=0$. Then if
$$2<p<p_c=1+\sqrt 2,\;\; \gamma=\frac{1}{2}-\frac{1}{p},
$$
and if
$$T_{\varepsilon',p}=c\varepsilon'^{\frac{p(p-1)}{p^2-2p-1}},
$$
then there exists an $\varepsilon_0>0$ depending on $\Omega$, $B$
and $p$ so that \eqref{eqn1} has an almost global solution in
$[0,T_{\varepsilon'}]\times\Omega$, satisfying $(Z^\alpha
u(t,\cdot),\partial_tZ^\alpha u(t,\cdot))\in
\dot{H}_B^\gamma\times\dot{H}_B^{\gamma-1}, \;|\alpha|\leq 2,\;
t\in[0,T_{\varepsilon'}]$, whenever the initial data satisfies the
boundary conditions of order $2$, and
\begin{equation}
\label{datasmall} \sum_{|\alpha|\leq 2}(\norm{Z^\alpha
f}_{\dot{H}_B^\gamma(\Omega)}+\norm{Z^\alpha
g}_{\dot{H}_B^{\gamma-1}(\Omega)})<\varepsilon'
\end{equation}
with $0<\varepsilon'<\varepsilon_0$.
\end{thm}

In the case where $\Omega=\Real^3$ and $\Delta_\g=\Delta$ it is
known that $p>p_c$ is necessary for global existence (see John
\cite{John}).  John [17] also established the global existence
theorem for $p>p_c=1+\sqrt 2$. For the local existence result,
Lindblad \cite{Lindblad} handled the case $1<p<1+\sqrt2$ in
$\Real^3$, then Zhou \cite{Zhou} obtained the case $p=1+\sqrt2$. In
their works it
was also shown that the lifespan estimates given are sharp.\\
\indent On the other hand, when the data is spherically symmetrical
and $n=3$, Sogge \cite{So} and Hidano \cite{Hidano} obtained the
sharp local well-posedness theorem for the Minkowski wave equation
respectively by using some radial estimates. It is also shown in
\cite{So} that
the regularity $\gamma=1/2-1/p$ is sharp for radial data.\\
\indent For nontrapping obstacles, Hidano, Metcalfe, Smith, Sogge
and Zhou \cite{KJHCY} dealt with the global existence part (i.e.
$p>p_c$) for \eqref{eqn1} with $n=3,4$.

Here we will use the real interpolation method to get the local
existence theorem for \eqref{eqn1} when the perturbation is
nontrapping. Before handling the obstacle problem we will first give
an alternative proof for the Minkowski space results, which involves
an interpolation between the following two estimates.
\begin{lem}
(A variant of the KSS estimate) For $n\geq 3$, let $u$ solve the
homogeneous wave equation \eqref{inhomowaveequation} ($F=0$) in the
Minkowski space. Then we have
\begin{equation}
\label{est7} \big\|\langle
x\rangle^ae^{it|D|}f\big\|_{L_t^2L_x^2([0,T]\times\mathbb{R}^n)}
\lesssim B(T)\|f\|_{\dot{H}^0},
\end{equation}
where
\begin{equation}
\label{B(T)}
B(T)= \left\{
\begin{array}{rll}
&T^{(1/2+a)},\; \;{\rm if}\;\;-1/2<a\leq 0 \\
&(\log((2+T)))^{1/2}, \;\; {\rm if}\;\;a=-1/2\\
& Constant, \;\; {\rm if}\;\;
 a<-1/2.
\end{array}
\right.
\end{equation}
In particular, for $-1/2<a\leq 0$, we have
\begin{equation}
\label{est1}
\big\||x|^ae^{it|D|}f\big\|_{L_t^2L_x^2([0,T]\times\mathbb{R}^n)}
\lesssim T^{(1/2+a)} \|f\|_{\dot{H}^0}.
\end{equation}
\end{lem}

\begin{proof}
 Actually the cases where $a\leq-1/2$ have been well set up in Du, Metcalfe, Sogge, Zhou ~\cite{MSYY}, and
can be adapted to handle the case $-1/2<a\leq0$. Specifically,
considering $u=-i\sum_{j=1}^{n}\partial_jv_j,
v_j=\check{F}[\hat{u}(t,\xi)\frac{\xi_j}{|\xi|^2}]$, the lemma
follows from
$$\big\|\langle x\rangle^a v'\big\|_{L_t^2L_x^2([0,T]\times\mathbb{R}^n)} \lesssim B(T)\|f\|_{\dot{H}^1},\;\;\; a\leq 0.$$
But this is from the following estimate
\begin{equation}
\label{kssest} \|v'(t,x)\|_{L_t^2L_x^2([0,\infty]\times{|x|<1})}
\lesssim \|v'(0,x)\|_{L^2_x},
\end{equation}
and a scaling argument for $|x|<T$, the energy inequality for
$|x|>T$. For the proof of \eqref{kssest}, refer to Keel, Smith and
Sogge \cite{KSS}.

As for \eqref{est1}, we just need to take care of the case where
$|x|<1$, but that is just a direct result of Lemma \ref{2.2} and a
scaling argument for a partition of $\{x: 0<|x|<1\}$. See
\cite{Hidano} for details.
\end{proof}

In what follows, we will employ \eqref{est1} to do the
interpolations for simplicity, while we remark that the weaker
estimate \eqref{est7} can actually lead us to the same existence
theorem as well by the same argument.

The next estimate is a result of the complex interpolation between
\eqref{est1} and the endpoint Trace Lemma.
\begin{prop}
For $n=3$, let $u$ solve the homogeneous wave equation
\eqref{inhomowaveequation} in Minkowski space. Then we have
\begin{equation}
\label{est2}
\big\||x|^{(1+2a)/3}e^{it|D|}f\big\|_{L_t^3L_r^3L_\omega^2([0,T]\times\mathbb{R}^3)}
\less T^{(1+2a)/3}\|f\|_{\dot{B}^{1/6}_{2,3/2}},
\end{equation}
for $-1/2<a\leq 0$.
\end{prop}

Here, and in what follows, we are using the mixed-norm notation with
respect to the volume element
$$
\|h\|_{L^q_rL^p_\omega}=\Bigl(\, \int_0^\infty \, \Bigl(\,
\int_{S^{n-1}}|h(r\omega)|^p\, d\sigma(\omega)\, \Bigr)^{q/p} \,
r^{n-1}dr\, \Bigr)^{1/q}
$$
for finite exponents and
$$
\|h\|_{L^\infty_rL^p_\omega}=\sup_{r>0} \Bigl(\,
\int_{S^{n-1}}|h(r\omega)|^p\, d\sigma(\omega)\, \Bigr)^{1/p}.
$$

Also, the homogeneous Besove space $\dot B_{p,q}^s$ is defined as
$$\|f\|_{\dot B_{p,q}^s}=\|2^{js}P_jf\|_{l^q_j(j\in\Z)L^p_x}, ~~\text{where}~ f=\sum_{j}P_jf ~\text{is the Littlewood-Paley decomposition.}$$

\begin{proof}
Recall that we have the endpoint Trace Lemma (see \cite{CF}):
$$
\big\||x|^{\frac{n-1}2}e^{it|D|}f\big\|_{L_t^\infty L_r^\infty
L_\omega^2([0,T]\times\mathbb{R}^n)} \less
\|f\|_{\dot{B}^{1/2}_{2,1}}.
$$
Now if we use the complex interpolation between this estimate and
\eqref{est1} for $n=3$, and set $\theta=1/3$. Noting $\dot
B^0_{2,2}=\dot H^0$ and using the fact of
$(\dot{B}^0_{2,2},\dot{B}^{1/2}_{2,1})_{[\theta]}=\dot{B}^{1/6}_{2,3/2}$
 for $\theta=1/3$(see Section 6.4 in \cite{JJ}), we get the desired estimate
\eqref{est2} for $-1/2<a\leq 0$.
\end{proof}

Next we will cite some notations and results in \cite{JJ} and
\cite{TH}. Let $A_0, A_1$ be Banach spaces, define the real
interpolation space $(A_0, A_1)_{\theta,q}$ for $0<\theta<1$ and $
1\leq q\leq \infty$ via the norm:
$$
\norm{a}_{(A_0, A_1)_{\theta,q}}=\norm{a}_{(A_0,
A_1)_{\theta,q;K}}=\Big(\int_0^\infty(t^{-\theta}K(t,a))^qdt/t\Big)^{1/q},
$$
where
$$
K(t,a)=\inf_{a=a_0+a_1}(\norm{a_0}_{A_0}+\norm{a_1}_{A_1}).
$$
Now if we let
$$\begin{aligned}
A_0&=\dot{B}^0_{2,2}, \;\;A_1=\dot{B}^{1/6}_{2,3/2},\\
B_0&=L_{t,r}^2L_\omega^2([0,T]\times[0,\infty)\times S^2, r^{2+2a}dtdrdw),\\
B_1&=L_{t,r}^3L_\omega^2([0,T]\times[0,\infty)\times S^2,
r^{3+2a}dtdrdw),
\end{aligned}
$$
then by \eqref{est1} and \eqref{est2}, we have
\begin{equation}
\label{est4} Tf=e^{it|D|}f: \bar{A}\rightarrow \bar{B}, \quad
\text{where}\ \bar A=(A_0,A_1), \bar B=(B_0,B_1),
\end{equation}
and
\begin{equation}
\label{est5} M_0\less T^{1/2+a},\; M_1\less T^{2/3(1/2+a)}, \quad
\text{where}\ M_i=\norm T_{A_j,B_j},\ j=0,1.
\end{equation}

Now we can state the main weighted Strichartz estimates as follows:
\begin{prop}
\label{3.4} For $n=3$, let $u$ solves the homogeneous wave equation
\eqref{inhomowaveequation} in Minkowski space ($F=0$). Then we have
\begin{equation}
\label{est3}
\big\||x|^{(-1/2-\gamma)/p}u\big\|_{L_t^pL_r^pL_\omega^2([0,T]\times\mathbb{R}^3)}
\less T^{(-p+1/p+2)/p}(\|f\|_{\dot{H}^{\gamma}}+\norm g_{\dot
H^{\gamma-1}}),
\end{equation}
where $\gamma=1/2-1/p$ and $2<p<1+\sqrt 2$.
\end{prop}

\begin{proof}
The result when the data is radial was shown in \cite{Hidano}. Here
we will use a different method to handle with the nonradial case.

Since $K_{\theta,q}$ is an exact interpolation functor of exponent
$\theta$ (Theorem3.1.2 in \cite{JJ}), from \eqref{est4} and
\eqref{est5} we get
\begin{equation}
\label{est6}
\begin{aligned}
\norm{Tf}_{\bar B_{\theta,2}}&\leq M_0^{1-\theta}M_1^{\theta}\norm f_{\bar A_{\theta,2}}\\
&\less T^{(1-\frac13\theta)(\frac12+a)}\norm f_{\bar A_{\theta,2}},
\end{aligned}
\end{equation}
if $-1/2<a\leq 0$. To proceed, we note that from Theorem 6.4.5 in
\cite{JJ} we have
$$(B_{pq_0}^{s_0}, B_{pq_1}^{s_1})_{\theta,r}=B_{pr}^{s^*},\ \text{if}\ s_0\neq s_1,~ 0<\theta<1,~ r,q_0,q_1\geq 1 ~\text{and}~ s^*=(1-\theta)s_0+\theta s_1.$$
Set $a=(-p+1/p+1)/2$ and $\theta=3-6/p$, then we have $0<\theta<1$
and $-1/2<a\leq 0$ since $2<p<1+\sqrt 2$. Thus we see
\begin{equation}
\label{RHS} RHS\  of\  \eqref{est6}\less T^{(-p+1/p+2)/p}\norm
f_{\dot B^{1/2-1/p}_{2,2}}=T^{(-p+1/p+2)/p}\norm f_{\dot
H^{1/2-1/p}}.
\end{equation}
On the other hand, we can use the fact (Theorem 3.4.1(b) in
\cite{JJ})
$$\bar A_{\theta,q}\subset \bar A_{\theta, r},\quad \text{if}\ q\leq r,$$
and bilinear weighted interpolation (Section 1.18.5 in \cite{TH})
$$\Big(\big(L_{t,r}^{p_0}L^2_\omega,w_0(r)dtdrdw\big),\big(L_{t,r}^{p_1}L^2_\omega,w_1(r)dtdrdw\big)\Big)_{\theta,p}=\big(L_{t,r}^{p}L^2_\omega,w(r)dtdrdw\big),$$
if $1/p=(1-\theta)/{p_0}+\theta /{p_1},\
w(r)=w_0^{p(1-\theta)/{p_0}}w_1^{{p\theta}/{p_1}}$.

Since $p>2$, we also have
\begin{equation}
\begin{aligned}
\label{LHS}
LHS\ of\ \eqref{est6}\gtrsim \norm{Tf}_{\bar B_{\theta,p}}&=\norm{Tf}_{(L_{t,r}^{p}L^2_\omega,r^{1+1/p}dtdrdw)}\\
&=\big\||x|^{\frac{-1+1/p}p}Tf\big\|_{L_t^pL_r^pL_w^2([0,T]\times\mathbb{R}^n)}.
\end{aligned}
\end{equation}
Now \eqref{est3} is just a result of \eqref{RHS} and \eqref{LHS}.
\end{proof}

As a result, by the arguments to follow, \eqref{est3} is strong
enough to show the local existence of solutions as described in
Theorem \ref{existthm1} in the Minkowski space case.

To prove the obstacle version of this result, we define
$X=X_{\gamma,p}(\Real^n)$ to be the space with norm defined by
\begin{equation}
\label{defineX} \norm h_{X_{\gamma,p}}= \norm
h_{L^{s_\gamma}(|x|<2R)}+(A(T))^{-1}\norm
{|x|^{{-{1}/{2}-\gamma}/{p}}h}_{L_r^pL_\omega^2(|x|>2R)},
\end{equation}
with $A(T)=T^{\frac{-p+\frac{1}{p}+2}{p}}$ and
$s_\gamma=\frac{2n}{n-2\gamma}$.

Using the space $X$ defined just now, we can prove the following
estimate provided $\gamma={1}/{2}-{1}/{p}$ and $p\geq 2$:
\begin{equation}
\label{3.11} \|u\|_{L_t^pX([0,T]\times \mathbb{R}^3)} \lesssim
\|f\|_{\dot{H}^\gamma}+\|g\|_{\dot{H}^{\gamma-1}},
\end{equation}
\noindent when $u$ solves $\Box u=0$ with initial data $(f,g)$.

Indeed, the contribution of the second part of the norm in
\eqref{defineX} is controlled by \eqref{est3}, and the contribution
of the first term is due to Sobolev estimates and an interpolation
between $L^2_t$ and $L^\infty_t$ in \eqref{energyest2}(Note that
$\three=0$ in our case).

Furthermore, by finite propagation speed of the wave equation,
Sobolev estimates and interpolation between \eqref{energyest2}, we
have the local estimate for solutions of \eqref{obwaveequation} with
$F=0$:
\begin{equation}
\label{3.12} \|u\|_{L_t^pX([0,1]\times \Omega)} \lesssim
(\|f\|_{\dot{H}^\gamma}+\|g\|_{\dot{H}^{\gamma-1}}),
\end{equation}
where $p\geq 2$.

From \eqref{3.11} and \eqref{3.12}, we see that $(X,\gamma,0,p)$ is
admissible. By Theorem \ref{mainest2}, we therefore obtain the
following Proposition:
\begin{prop}
\label{prop3.5}
 For $n=3$, let $u$ be a solution of \eqref{obwaveequation} with
$F=0$, and let $\Omega$ be such a domain as described in the
beginning of this section. Moreover, assume that
\begin{equation}
\label{3.13}\gamma=\frac12-\frac1p, \ p\in(2,1+\sqrt2).
\end{equation}
Then
\begin{equation}
\label{prop3.6} \|u\|_{L_t^pX([0,T]\times \Omega)} \lesssim
\|f\|_{\dot{H}^\gamma}+\|g\|_{\dot{H}^{\gamma-1}}.
\end{equation}
\end{prop}
From the above proposition we get the following useful corollary.
\begin{cor}
For $n=3$, let $u$ be a solution of \eqref{obwaveequation}, and let
$\Omega$ be such a domain as described in the beginning of this
section. Moreover, assume the condition \eqref{3.13}. Then
\begin{eqnarray}
\label{cor3.6} \norm
{u}_{L_t^pL_x^{s_\gamma}([0,T]\times|x|<2R)}+(A(T))^{-1}\norm
{|x|^{(-{1}/{2}-\gamma)/{p}}u}_{L_t^pL_r^pL_\omega^2(|x|>2R)}\nonumber\\
&\hspace{-2.5in} \less
\|f\|_{\dot{H}^\gamma}+\|g\|_{\dot{H}^{\gamma-1}}+\norm
F_{L_t^1L_x^{s_{1-\gamma}'}([0,T]\times\{|x|<2R\})}\nonumber\\
&\hspace{-2in}+\norm{|x|^{-{1}/{2}-\gamma}F}_{L_t^1L_r^1L_\omega^2([0,T]\times
\{|x>2R|\})}.
\end{eqnarray}
\end{cor}

\begin{proof}
This is an immediate consequence of Duhamel's principle, Sobolev
estimates and the following estimate (originated in \cite{LiZhou},
see also (3.7)in \cite{KJHCY}):
$$\norm \varphi_{\dot{H}^{\gamma-1}}\less \norm{|x|^{-{n}/{2}+1-\gamma}\varphi}_{L_r^1L_\omega^2}, \quad {\rm if}\; \frac{1}{2}
<1-\gamma<\frac{n}{2}.$$ Here the condition ${1}/{2}
<1-\gamma<{n}/{2}$ is satisfied owing to \eqref{3.13}.
\end{proof}

If we set $\Gamma=\{\partial_t,Z\}$, then we can easily adapt such
an argument as in \cite{KJHCY} (see page 15-17) to get the following
higher order estimates of \eqref{cor3.6} and \eqref{energyest2}:
\begin{equation}
\label{estimate1}
\begin{aligned}
\sum_{|\alpha|\leq 2}(\norm {\Gamma^\alpha
u}_{L_t^pL_x^{s_\gamma}([0,T]\times\{|x|<2R\})}+A(T)^{-1}\norm
{|x|^{{-{1}/{2}-\gamma}/{p}}\Gamma^\alpha u}_{L_t^pL_r^pL_\omega^2([0,T]\times\{|x|>2R\})})\\
&\hspace{-2.5in} \less
\sum_{|\alpha|\leq 2}(\|Z^\alpha
f\|_{\dot{H}^\gamma}+\|Z^\alpha
g\|_{\dot{H}^{\gamma-1}})+\sum_{|\alpha|\leq 2}(\norm
{\Gamma^\alpha F}_{L_t^1L_x^{s_{1-\gamma}'}([0,T]\times\{|x|<2R\})}\\
&\hspace{-2in}+\norm{|x|^{-{1}/{2}-\gamma}\Gamma^\alpha
F}_{L_t^1L_r^1L_\omega^2([0,T]\times \{|x>2R|\})}).
\end{aligned}
\end{equation}

\begin{equation}
\label{estimate2}
\begin{aligned}
\sum_{|\alpha|\leq 2}(\norm {\Gamma^\alpha u}_{L_t^\infty
\dot{H}_B^\gamma([0,T]\times\Omega)}+\norm
{\partial_t \Gamma^\alpha u}_{L_t^\infty \dot{H}_B^{\gamma-1}([0,T]\times \Omega)})\\
&\hspace{-2.5in} \less \sum_{|\alpha|\leq 2}(\|Z^\alpha
f\|_{\dot{H}^\gamma}+\|Z^\alpha
g\|_{\dot{H}^{\gamma-1}})+\sum_{|\alpha|\leq 2}(\norm
{\Gamma^\alpha F}_{L_t^1L_x^{s_{1-\gamma}'}([0,T]\times\{|x|<2R\})}\\
&\hspace{-2in}+\norm{|x|^{-{1}/{2}-\gamma}\Gamma^\alpha
F}_{L_t^1L_r^1L_\omega^2([0,T]\times \{|x>2R|\})}).
\end{aligned}
\end{equation}

Now we set
$$
M_k(T)=\sum_{|\alpha|\leq 2}\Big(\norm {\Gamma^\alpha
u_k}_{L_t^pL_x^{s_\gamma}([0,T]\times\{|x|<2R\})}+A(T)^{-1}\norm
{|x|^{{-{1}/{2}-\gamma}/{p}}\Gamma^\alpha
u_k}_{L_t^pL_r^pL_\omega^2([0,T]\times\{|x|>2R\})}\Big),
$$
where $u_k, k\geq 0$ is the solution of
\begin{equation}
\label{eqn3} \left\{
\begin{array}{rll}
&(\partial_t^2-\Delta_g)u_k = F_p(u_{k-1}(t,x))\;{\rm on}\;\;\mathbb{R}_+\times \Omega  \\
&u_k|_{t=0}= f,  \\
&\partial_t u_k|_{t=0}= g,  \\
&(Bu_k)(t,x)=0,\;{\rm on}\;\;\mathbb{R}_+\times \partial\Omega.\\
\end{array}
\right.\\
\end{equation}

By the same iteration argument as followed in Section 4, we obtain
Theorem
~\ref{existthm1}.\\

\textbf{Note.} If we use the KSS estimate for $a=-1/2$ instead of
$-1/2<a\leq 0$, and the same complex interpolation method and the
same real interpolation method as above, we will get Proposition
\ref{3.4} with $p=1+\sqrt 2, a=-1/2$ and
$A(T)=(log(2+T))^{{1}/{p}}$. Furthermore we get the
local-wellposedness for the critical power $p=1+\sqrt{2}$ with
$T_\varepsilon=exp(C\varepsilon^{-(p-1)})$, but this life span is
not optimal (The optimal one should be
$T_\varepsilon=\text{exp}(C\varepsilon^{-p(p-1)}))$.

%%%%%%%%%%%%%%%%%%%%%%%%%%%%%%%%%%%%%%%%%%%%%%%%%%%%%%%%%%%%%%%%%%%%%%%%%%%%%%%%%%%%%%%%%%%%%%%%%%%%%%%%%%%%%%%%%%%%%%%%%

\newsection{Application 2: Strauss conjecture on semilinear wave equations with finitely many obstacles}

We will consider wave equations of the form
\begin{equation}\label{obeq2}
\begin{cases}
(\partial_t^2-\Delta_\g)u(t,x)=F_p\bigl(u(t,x)\bigr),
\quad (t,x)\in \Real_+\times \Omega
\\
Bu=0, \quad \text{on } \, \Real_+\times \partial\Omega
\\
u(0,x)=f(x), \quad \partial_t u(0,x)=g(x), \quad x\in \Omega,
\end{cases}
\end{equation}
with $B$ described as in the first section.
$\Omega=\Real^n\backslash \bigcup_{i=1}^m\kappa_i$ where $\kappa_i
(i=1,2,\cd,m )$ are disjoint compact convex subsets of $|x|<R$ with
smooth boundary. We will assume that the nonlinear term behaves like
$|u|^p$ when $u$ is small, and so we assume that
\begin{equation}\label{1.4}
\sum_{0\le j\le 2} |u|^j\, \bigl|\, \partial^j_u F_p(u)\, \bigr| \,
\lesssim \, |u|^p,
\end{equation}
when $u$ is small.

Ikawa \cite{Ikawa} managed to show that solutions of \eqref{obeq2}
with $n=3, ~\Delta_g=\Delta$, $B=I$, and $F_p(u)=0$ have exponential
decay estimates with a loss of 2 derivatives of data. To assure
this, we need some technical assumptions on the obstacles (see page
3-4 in \cite{Ikawa}), which we will assume are satisfied here. Now,
interpolating between this estimate and the energy estimate we get
an estimate of the form:
$$\norm {u'(t,x)}_{L^2_x(|x|<1)}\less e^{-ct}\norm {u'(0,x)}_{\dot H^\three(|x|<1)},
\quad \text{for any positive number}\ \three.
$$
This motivates us to study \eqref{obeq2} under Hypothesis B.

In the next theorem we are abusing Hypothesis B a little by assuming
it is true for $n=4$. Actually there has been no polynomially local
energy decay set up for even dimensions when there are trapped rays,
which could be expected though. And Burq did show that local energy
decays at least logarithmically with some loss in
derivatives(\cite{Burq1}).
\begin{thm}
\label{existthm}  Let $n=3$ or $4$, and fix $\Omega\subset \Rn$ and
boundary operator $B$ as above.  Assume further that Hypothesis B is
valid for an arbitrarily small $\three>0$.

Let $p=p_c$ be the positive root of
\begin{equation}\label{1.5}
(n-1)p^2-(n+1)p-2=0.
\end{equation}
If
\begin{equation}\label{gamma}
p_c<p< (n+3)/(n-1),\ \gamma=\tfrac{n}2-\tfrac2{p-1},
\end{equation}
then there exists a $\varepsilon_0>0$ depending on $\Omega, B$ and
$p$ and $\three$  so that \eqref{obeq2} has a global solution
satisfying $(Z^\alpha u(t,\cd), \partial_t Z^\alpha u(t,\cd))\in
\dot H^\gamma_B\times \dot H^{\gamma-1}_B$, $|\alpha|\le 2$, $t\in
\Real_+$, whenever the initial data satisfies the boundary
conditions of order $2$, and
\begin{equation}\label{databd}
\sum_{|\alpha|\le 2}\Bigl(\, \|Z^\alpha f\|_{\tilde H_{2\three}^\gamma(\Omega)}
+\|Z^\alpha g\|_{\tilde H_{2\three}^{\gamma-1}(\Omega)}\, \Bigr)<\varepsilon'
\end{equation}
with $0<\varepsilon'<\varepsilon_0$.

On the other hand, if
\begin{equation}
\label{gamma2} n=3, \;\gamma=\frac12-\frac1p,\; p\in(2,1+\sqrt2)
\end{equation}
and
\begin{equation}
T_{\three'}= c\three'^{\frac{p(p-1)}{p^2-2p-1}},
\end{equation}
then there exists a unique solution in $[0,T_{\three'})\times\Omega$
such that $(Z^\alpha u(t,\cd), \partial_t Z^\alpha u(t,\cd))\in\dot
H^\gamma_B\times \dot H^{\gamma-1}_B$ under the condition
\eqref{databd}.
\end{thm}

Before we turn to the proof of this existence theorem, we will first
employ Theorem \ref{mainest2} to get important estimates that will
be used.

Define $X=X_{\gamma,p}(\Real^n)$ to be the space with the norm
defined by
\begin{equation}
\label{X} \norm h_{X_{\gamma,p}}= \norm
h_{L^{s_\gamma}(|x|<2R)}+(A(S))^{-1}\norm
{|x|^{{-{n}/{2}+1-\gamma}/{p}}h}_{L_r^pL_\omega^2(\{|x|>2R\})},
\end{equation}
where $s_\gamma={2n}/{n-2\gamma}$. When $n=3, p< p_c,
\gamma=\frac12-\frac1p$, we have $S=T$ and $A(T)$ is as defined in
the last section; When $n=3,4,~p>p_c$ and $\gamma=n/2-2/{p-1}$ we
have $S=\infty$ and $A(S)$ is a constant.

Now by using \eqref{est3}, a known result (3.6) in \cite{KJHCY} and
energy estimates, we can adapt the argument in Section 3 to get the
following proposition:
\begin{prop}
 For $n=3$ or $4$, let $u$ be a solution of \eqref{obwaveequation} with
$F=0$, and assume condition \eqref{gamma} or \eqref{gamma2} is
satisfied. Then
\begin{equation}
\label{prop11} \|u\|_{L_t^pX([0,S]\times \Omega)} \lesssim
\|f\|_{\tilde{H}_{\three}^\gamma}+\|g\|_{\tilde{H}_{\three}^{\gamma-1}}.
\end{equation}
\end{prop}

Based on the above proposition, it is easy to get the following
corollary with forcing term added.
\begin{cor}
\label{cor4.3} For $n=3,4$, let $u$ be a solution of
\eqref{obwaveequation}, and assume condition \eqref{gamma} or
\eqref{gamma2} is satisfied. Then
\begin{multline}
\label{est11} \norm
{u}_{L_t^pL_x^{s_\gamma}([0,S]\times\{|x|<2R\})}+(A(S))^{-1}\norm
{|x|^{{-{n}/{2}+1-\gamma}/p}u}_{L_t^pL_r^pL_\omega^2([0,S]\times \{|x|>2R\})}\less
\|f\|_{\tilde{H}_{\three}^\gamma}+\|g\|_{\tilde{H}_{\three}^{\gamma-1}}\\
 +\norm
F_{L_t^1L_x^{s_{1-\gamma-\three}'}(\Real_+\times\{|x|<2R\})}+\norm{|x|^{-{n}/{2}+1-\gamma}F}_{L_t^1L_r^1L_\omega^2(\Real_+
\times\{|x|>2R\})}.
\end{multline}
\end{cor}

\begin{proof}
By \eqref{prop11}, we can assume $f=g=0$. By Duhamel's principle, we
have
$$\begin{aligned}LHS&\less \norm F_{L_t^1\tilde{H}_{\three}^{\gamma-1}(\Real_+\times\Omega)}\\
&\less \norm F_{L_t^1\dot{H}^{\gamma-1}(\Real_+\times\Omega)}+\norm
F_{L_t^1\dot{H}^{\gamma+\three-1}(\Real_+\times\Omega)}.
\end{aligned}
$$
Recall that the dual version of the trace lemma and Sobolev
embedding gives (see (3.16) of \cite{KJHCY}):
\begin{equation}
\label{3.16} \norm g_{\dot{H}^{\gamma-1}}\less
\norm{|x|^{-{n}/{2}+1-\gamma}g}_{L_r^1L_\omega^2(\{|x|>2R\})}+\norm
g_{L^{s_{1-\gamma}'}(\{|x|<2R\})}, \quad {\rm if}\;
{1}/{2}<1-\gamma<{n}/{2}.
\end{equation}
Here the condition ${1}/{2} <1-\gamma<{n}/{2}$ is satisfied owing to
\eqref{gamma} or \eqref{gamma2}.\\
If we use \eqref{3.16}, then we get
$$\begin{aligned}
\norm F_{L_t^1\dot{H}^{\gamma-1}([0,S]\times\Omega)}&+\norm
F_{L_t^1\dot{H}^{\gamma+\three-1}([0,S]\times\Omega)}\less\\
&\norm{|x|^{-{n}/{2}+1-\gamma}F}_{L_t^1L_r^1L_\omega^2([0,S]\times\{|x|>2R\})}
+\norm F_{L_t^1L_x^{s_{1-\gamma}'}([0,S]\times\{|x|<2R\})}\\
&+\norm{|x|^{-{n}/{2}+1-\gamma-\three}F}_{L_t^1L_r^1L_\omega^2([0,S]\times\{|x|>2R\})}+\norm F_{L_t^1L_x^{s_{1-\gamma-\three}'}([0,S]\times\{|x|<2R\})}\\
&\less
\norm{|x|^{-{n}/{2}+1-\gamma}F}_{L_t^1L_r^1L_\omega^2([0,S]\times\{|x|>2R\})}+\norm
F_{L_t^1L_x^{s_{1-\gamma-\three}'}([0,S]\times\{|x|<2R\})},
\end{aligned}
$$
when $\three>0$ is small enough, which completes the proof.\\
\end{proof}

By modifying the proof of corresponding estimates in \cite{KJHCY},
we get the following higher order estimates of \eqref{energyest2}
and \eqref{est11}, which are key to prove the existence theorem.
\begin{prop}
$($Higher order Energy and Strichartz Estimates$)$. Suppose that
data $(f,g,F)$ satisfies the $H_B^2\times H^1_B\times H_B^1$
boundary conditions. Under the conditions in Corollary \ref{cor4.3},
the following estimates hold:

\begin{align}\label{4.14}
&\sum_{|\alpha|\le 2}
\Bigl( \|\, \Gamma^\alpha u
\|_{L^\infty_t \dot{H}_B^\gamma}
+\|\partial_t\Gamma^\alpha u\|_{L^\infty_t\dot{H}_B^{\gamma-1}}
\Bigr)\lesssim \sum_{|\alpha|\le 2}
\Bigl( \|Z^\alpha f\|_{\tilde H^\gamma_{2\three}}
+\|Z^\alpha g\|_{\tilde H^{\gamma-1}_{2\three}}\Bigr)
\\
&+\sum_{|\alpha|\le 2} \Bigl(\|\, |x|^{-\frac{n}2+1-\gamma}
\Gamma^\alpha F\|_{L^1_tL^1_rL^2_\omega(\Real_+\times \{|x|>2R\})}+
\|\Gamma^\alpha F\|_{L^1_tL_x^{s'_{1-\gamma-2\three}} (\Real_+\times
\{x\in \Omega: |x|<2R\})}\Bigr),\notag
\end{align}
and
\begin{align}\label{4.15}
&\sum_{|\alpha|\le 2}
\Bigl( \|\, |x|^{\frac{n}2-\frac{n+1}p-\gamma} \Gamma^\alpha u
\|_{L^p_tL^p_rL^2_\omega(\Real_+\times \{|x|>2R\})}
+\|\Gamma^\alpha u\|_{L^p_tL^{s_\gamma}_x(\Real_+\times \{x\in \Omega: |x|<2R\})}
\Bigr)
\\
&\lesssim \sum_{|\alpha|\le 2}
\Bigl( \|Z^\alpha f\|_{\tilde H^\gamma_{2\three}}
+\|Z^\alpha g\|_{\tilde H^{\gamma-1}_{2\three}}\Bigr)
\notag
\\
&+\sum_{|\alpha|\le 2}
\Bigl(\|\, |x|^{-\frac{n}2+1-\gamma}
\Gamma^\alpha F\|_{L^1_tL^1_rL^2_\omega(\Real_+\times \{|x|>2R\})}+
\|\Gamma^\alpha F\|_{L^1_tL_x^{s'_{1-\gamma-2\three}}
(\Real_+\times \{x\in \Omega: |x|<2R\})}\Bigr).\notag
\end{align}
\end{prop}

\begin{proof}
We will first deal with the Cauchy data for $\Gamma^\alpha u$. This
is clear if $\Gamma^\alpha$ is replaced by $Z^\alpha$. On the other
hand, the Cauchy data is $(g,\Delta_g f+F(0,\cd))$ for $\partial_t
u$ and $(\Delta_g f+F(0,\cd), \Delta_g g+\partial_t F(0,\cd))$ for
$\partial_t^2 u$, so we have
\begin{eqnarray*}
\|g\|_{\tilde H^\gamma_\three}+ \|\Delta_\g f\|_{\tilde
H^{\gamma-1}_\three\cap \tilde
H^{\gamma}_\three}+\|F\|_{L^\infty_t\tilde H^{\gamma-1}_\three\cap
L^\infty_t\tilde
H^{\gamma}_\three}+\|\partial_tF\|_{L^\infty_t\tilde
H^{\gamma-1}_\three}&+&\|\Delta_g g\|_{\tilde H^{\gamma-1}_\three}\less\\
\sum_{|\alpha|\le 2} \Bigl(\,\|Z^\alpha f\|_{\tilde H^\gamma_\three}
+\|Z^\alpha g\|_{\tilde
H^{\gamma-1}_\three}\Bigr)&+&\sum_{|\alpha|\le 2} \|\Gamma^\alpha
F\|_{L^1_t\tilde H^{\gamma-1}_\three}\,,
\end{eqnarray*}
where we use Sobolev embedding in the time variable $t$ for $(F,
\partial_t F)$. If we use \eqref{3.16} to control the last term $\sum_{|\alpha|\le 2} \|\Gamma^\alpha
F\|_{L^1_t\tilde H^{\gamma-1}_\three}$, then we get \eqref{4.14} and
\eqref{4.15} for the Cauchy data part of $\Gamma u$.

Let us now give the argument for \eqref{4.15}.  Fix $\beta_0\in
C^\infty_0$ satisfying $\beta_0=1$ for $|x|\le R$ and vanishing for
$|x|>2R$.  Let
$$
\Gamma^\alpha u=(1-\beta_0)\Gamma^\alpha u+\beta_0\Gamma^\alpha u=v+w.
$$
Since $\Gamma$ commutes with $\square_g$ when $|x|\ge R$, we have
$$
\begin{cases}
\square_g  v = (1-\beta_0)\Gamma^\alpha
F-[\beta_0,\Delta_\g]\Gamma^\alpha u\,,\\
v(0,\cd)=((1-\beta_0)\Gamma^\alpha u(0,\cd),\;
\partial_tv(0,\cd)=\partial_t (1-\beta_0) \Gamma^\alpha u(0,\cd).
\end{cases}
$$
The initial data has been taken care of from the discussion above,
and the first nonlinear term is dominated by the right hand side of
\eqref{4.15} by \eqref{est11}. For the second nonlinear term, we use
Proposition \ref{prop2.11} and control it by
\begin{equation}\label{3.20}
\sum_{|\alpha|\le 2} \|\, [\beta_0,\Delta_\g] \Gamma^\alpha
u\|_{L^2_tH^{\gamma-1}_B} \lesssim \sum_{j\le
2}\|\beta_1\partial^j_t u\|_{L^2_tH^{\gamma+2-j}_B},
\end{equation}
assuming that $\beta_1$ equals one on the support of $\beta_0$ and
is supported in  $R<|x|< 2R$. Note that $[\Box_g, \partial_t^2]=0$,
if we use \eqref{energyest2} for $\partial_t^2 u$ and Duhamel's
principle for the forcing term $\partial_t^2 F$, we can control
$\|\beta_1\partial_t^2 u\|_{L^2_tH^{\gamma}_B}$ by the right hand
side of \eqref{4.15}. On the other hand, by Cauchy-Schwarz and
Parseval's Formula,
$$
\|\beta_1\partial_t u\|^2_{L^2_tH^{\gamma+1}_B} \lesssim
\|\beta_1\partial^2_t u\|_{L^2_tH^{\gamma}_B}\, \|\beta_1
u\|_{L^2_tH^{\gamma+2}_B}.
$$
So it suffices to dominate $\|\beta_1 u\|_{L^2_tH^{\gamma+2}_B}$. By
elliptic regularity of the operator $\Delta_g$, we have
\begin{align*}
\|\beta_1 u\|_{L^2_tH^{\gamma+2}_B}&\lesssim
\|\beta_2\Delta_\g u\|_{L^2_tH^{\gamma}_B}+
\|\beta_2 u\|_{L^2_tH^{\gamma}_B}\\
& \lesssim \|\beta_2\partial_t^2 u\|_{L^2_tH^{\gamma}_B}+ \|\beta_2
u\|_{L^2_tH^{\gamma}_B}+ \|\beta_2 F\|_{L^2_tH^{\gamma}_B},
\end{align*}
where $\beta_2\in C_0^\infty$ equals one on support of $\beta_1$ and
is supported in the set where $|x|<2R$. The first two terms are
dominated as above using \eqref{energyest2} and Duhamel's principle.
For the last term, Sobolev embedding and duality yields
\begin{align}\label{3.21}
\|\beta_2 F\|_{L^2_tH^{\gamma}_B}
&\lesssim\sum_{|\alpha|\le 1}\|\partial_x^\alpha F
\|_{L^2_tL^{s'_{1-\gamma}}(\Real_+\times\{x\in\Omega:|x|\le 2R\})}\\
&\lesssim\sum_{|\alpha|\le 2}\|\partial_{t,x}^\alpha F
\|_{L^1_tL^{s'_{1-\gamma-\three}}(\Real_+\times\{x\in\Omega:|x|\le 2R\})}.
\notag
\end{align}

Thus we are done with the proof of
\eqref{4.15} when $\Gamma^\alpha u$ is replaced by $v$.

For $w=\beta_0 \Gamma^\alpha u$, the coefficients of $\Gamma$ are
bounded on support of $\beta_0$, so by Sobolev embedding
\begin{align*}
\sum_{|\alpha|\le 2}\;
\|\beta_0\Gamma^\alpha u
\|_{L^p_tL^{s_\gamma}_x(\Real_+\times\Omega)}
&\lesssim\sum_{|\alpha|\le 2}\;\|\beta_1\Gamma^\alpha u\|_{L^p_t\dot{H}^{\gamma}_B}
\\
&\lesssim\sum_{|j|\le 2}\;\Bigl(\|\beta_1\partial^j_t u\|_{L^2_tH^{\gamma+2-j}_B}
+\|\beta_1\Gamma^j u\|_{L^\infty_t\dot{H}^{\gamma}_B}\Bigr).
\end{align*}
The first term is dominated as above, and the bound for the second
term comes from \eqref{4.14}, so we are done with proof of
\eqref{4.15}.

Now we turn to the proof of \eqref{4.14}.

As before we first consider the inequality where $\Gamma^\alpha u$
is replaced by $v=(1-\beta_0)\Gamma^\alpha u$ in \eqref{4.14}. The
inequality involving initial data has been taken care of in the
first paragraph of the proof, and the first nonlinear term is from
energy estimates in $\Rn$, Duhamel's principle and \eqref{3.16}. For
the remaining term by \eqref{energyest2} we see that it is
controlled by
\begin{equation}\label{3.200}
\sum_{|\alpha|\le 2} \|\, [\beta_0,\Delta_\g] \Gamma^\alpha
u\|_{L^2_tH^{\gamma+\three-1}_B} \lesssim \sum_{j\le
2}\|\beta_1\partial^j_t u\|_{L^2_tH^{\gamma+\three+2-j}_B}.
\end{equation}
By almost the same argument as above we get the desired bound in
\eqref{4.14}.

Now we are only left with $w=\beta_0\Gamma^\alpha u$. First notice
that the left hand side of \eqref{4.14} with $w$ is dominated by
$\sum_{j\le 3}\|\beta_1\partial_t^j
u\|_{L^\infty_tH^{2+\gamma-j}_B}\,.$ For the case $j=0,1$, since
$$\begin{cases}
\Box_\g(\beta_1 u)=\beta_1 F+[\Delta_\g,\beta_2] u\\
(\beta_1u,\partial_t\beta_1u)|_{t=0}=(\beta_1f,\beta_1g),
\end{cases}
$$
we use \eqref{energyest1} with the Duhamel formula to bound
\begin{multline*}
\|\beta_1 u\|_{L^\infty_tH^{\gamma+2}_B}+
\|\beta_1 \partial_t u\|_{L^\infty_tH^{\gamma+1}_B}\\
\lesssim
\|\beta_1 f\|_{H^{\gamma+2}_B}+
\|\beta_1 g\|_{H^{\gamma+1}_B}+
\|\beta_2 u\|_{L^2_tH^{\gamma+\three+2}_B}+
\|\beta_1 F\|_{L^1_tH^{\gamma+\three+1}_B}.
\end{multline*}
The term on the right involving $u$ was controlled previously; on
the other hand, by Sobolev embedding,
$$
\|\beta_1 F\|_{L^1_tH^{\gamma+\three+1}_B}\lesssim\sum_{|\alpha|\le 2}
\|\partial^\alpha_x F\|_{L^1_tL_x^{s'_{1-\gamma-\three}}}\,.
$$
To handle the terms for $j=2,3$ we use the equation to bound
$$
\sum_{j=2,3}\|\beta_1\partial_t^j u\|_{L^\infty_tH^{2+\gamma-j}_B}\le
\sum_{j=0,1}\Bigl(\|\beta_1\partial_t^j\Delta_\g u\|_{L^\infty_tH^{\gamma-j}_B}+
\|\beta_1\partial_t^j F\|_{L^\infty_tH^{\gamma-j}_B}\Bigr).
$$
The terms involving $\Delta_\g u$ are dominated by
$\|\beta_2\partial_t^j u\|_{L^\infty_tH^{\gamma+2-j}_B}$ with
$j=0,1$. The terms involving $F$ are controlled for $j=1$ by Sobolev
Embedding Theorem, and for $j=0$ by observing that \eqref{3.21}
holds with $L^2_t$ replaced by $L^\infty_t$. This completes the
proof of \eqref{4.14}.
\end{proof}

%%%%%%%%%%%%%%%%%%%%%%%%%%%%%%%%%%%%%%%%%%%%%%%%%%%%%%%%%%%%%%
\noindent\textbf{Proof of Theorem \ref{existthm}}:\\
We will adapt the argument from \cite{KJHCY}.
First, let $u_0$ solve the Cauchy problem \eqref{obwaveequation} with $F=0$.
We iteratively define $u_k$, for $k\ge 1$, by solving
$$
\begin{cases}
(\partial^2_t-\Delta_\g)u_k(t,x)=F_p(u_{k-1}(t,x))\,,
\quad (t,x)\in \Real_+\times \Omega
\\
u_k(0,\cd)=f,\;  \partial_t u_k(0,\cd)=g\\
(Bu_k)(t,x)=0,\quad \text{on } \, \Real_+\times \partial \Omega.
\end{cases}
$$
Our aim is to show that if the constant $\varepsilon'>0$ in \eqref{databd} is
small enough, then so is
\begin{multline*}
M_k = \sum_{|\alpha|\le 2}\,
\Bigl( \,
\bigl\|\Gamma^\alpha u_k\bigr\|_{L^\infty_t\dot H^\gamma_B([0,S]\times\Omega)}+
\bigl\|\partial_t\Gamma^\alpha u_k
\bigr\|_{L^\infty_t\dot H^{\gamma-1}_B([0,S]\times\Omega)}
\\+
(A(S))^{-1}\bigl\|\, |x|^{\frac{-\frac{n}{2}+1-\gamma}p} \Gamma^\alpha u_k
\bigr\|_{L^p_tL^p_rL^2_\omega([0,S]\times \{|x|>2R\})}
+\|\Gamma^\alpha u_k
\|_{L^p_tL^{s_\gamma}_x([0,S]\times \{x\in \Omega: \, |x|<2R\})}\, \Bigr)
%+\|\Gamma^\alpha u_k\|_{L^\infty_tL^{s_\gamma}_x(\R_+\times \Omega)},
\end{multline*}
for every $k=0,1,2,\dots$.

For $k=0$, it follows by \eqref{4.14} and \eqref{4.15} that $M_0\le
C_0\varepsilon'$, with $C_0$ a fixed constant. More generally,
\eqref{4.14} and \eqref{4.15} yield that
\begin{align}\label{3.23}
M_k\le C_0\varepsilon' +C_0\sum_{|\alpha|\le 2} \,
\Bigl( \,& \bigl\| \, |x|^{-\frac{n}2+1-\gamma}
\Gamma^\alpha F_p(u_{k-1})
\bigr\|_{L^1_tL^1_rL^2_\omega(\Real_+\times \{|x|>2R\})}
\\
&+\|\Gamma^\alpha F_p(u_{k-1})
\|_{L^1_tL^{s_{1-\gamma-2\three}'}_x(\Real_+\times \{x\in \Omega: \,
|x|<2R\})}\Bigr)\,. \notag
\end{align}
Note that our assumption \eqref{1.4} on the nonlinear term $F_p$ implies that
for small $v$
$$
\sum_{|\alpha|\le 2}|\Gamma^\alpha F_p(v)|\lesssim |v|^{p-1}
\sum_{|\alpha|\le 2}|\Gamma^\alpha v|+|v|^{p-2}
\sum_{|\alpha|\le 1}|\Gamma^\alpha v|^2\,.
$$
Furthermore, since $u_k$ will be locally of regularity
$H_B^{\gamma+2}\subset L^\infty$ and
$F_p$ vanishes at $0$, it follows that $F_p(u_k)$ satisfies the
$B$ boundary conditions if $u_k$ does.\\
Since the collection $\Gamma$ contains vectors spanning the tangent space
to $S^{n-1}$, by Sobolev embedding for $n=3,4$ we have
$$
\|v(r\cd)\|_{L^\infty_\omega}
+\sum_{|\alpha|\le 1}\|\Gamma^\alpha v(r\cd)\|_{L^4_\omega}
\lesssim \sum_{|\alpha|\le 2}
\|\Gamma^\alpha v(r\cd)\|_{L^2_\omega}\,.
$$
Consequently, for fixed $t, r>0$
$$
\sum_{|\alpha|\le 2}\|\Gamma^\alpha F_p(u_{k-1}(t,r\cd) )
\|_{L^2_\omega}\lesssim \sum_{|\alpha|\le 2} \|\Gamma^\alpha u_{k-1}(t,r\cd)
\|^p_{L^2_\omega}\,.
$$
Thus the first summand in the right hand side of \eqref{3.23} is
dominated by
$C_1\big(A(S)M_{k-1}\big)^p\,.$\\
We next observe that, since $s_\gamma>2$ and $n\le 4$,
it follows by Sobolev embedding on $\{\Omega\cap|x|<2R\}$ that
$$
\|v\|_{L^\infty(x\in\Omega:|x|<2R)}
+\sum_{|\alpha|\le 1}\|\Gamma^\alpha v\|_{L^4(x\in\Omega:|x|<2R)}
\lesssim \sum_{|\alpha|\le 2}
\|\Gamma^\alpha v\|_{L^{s_\gamma}(x\in\Omega:|x|<2R)}\,.
$$
Since $s_{1-\gamma-2\three}'<2$, it holds for each fixed $t$ that
\begin{multline}
\sum_{|\alpha|\le 2}\|\Gamma^\alpha F_p(u_{k-1}(t,\cd))
\|_{L^{s'_{1-\gamma-2\three}}(x\in\Omega:|x|<2R)}\less\sum_{|\alpha|\le
2}\|\Gamma^\alpha F_p(u_{k-1}(t,\cd))
\|_{L^{2}(x\in\Omega:|x|<2R)}\\
\lesssim \sum_{|\alpha|\le 2} \|\Gamma^\alpha u_{k-1}(t,\cd)
\|^p_{L^{s_\gamma}(x\in\Omega:|x|<2R)}\,.
\end{multline}
The second summand in the right side of \eqref{3.23} is thus
dominated by $C_1M_{k-1}^p\,,$ and we conclude that $M_k\le
C_0\epsilon'+2C_0\,C_1(A(S)M_{k-1})^p$. For $\epsilon'$ sufficiently
small, by the definition of $A(S)$, we obtain
\begin{equation}\label{3.24} M_k\le 2\,C_0\varepsilon', \quad k=1,2,3,\dots
\end{equation}\\
To finish the proof of Theorem \ref{existthm} we need to show that
$u_k$ converges to a solution of the equation
\eqref{obeq2}.  For this it suffices to show that
\begin{multline*}A_k=
(A(S))^{-1}\bigl\| \, |x|^{\frac{-\frac{n}{2}+1-\gamma}p}(u_k-u_{k-1})\,
\bigr\|_{L^p_tL^p_rL^2_\omega
([0,S]\times \{|x|>2R\})}
\\
+\|u_k-u_{k-1}\|_{L^p_tL^{s_\gamma}_x([0,S]\times \{x\in \Omega: \, |x|<2R\})}
\end{multline*}
tends geometrically to zero as $k\to \infty$.  Since
$|F_p(v)-F_p(w)|\lesssim |v-w|(\, |v|^{p-1}+|w|^{p-1}\, )$ when $v$ and
$w$ are small, the proof of \eqref{3.24} can be adapted to show that,
for small $\varepsilon'>0$, there is a uniform constant $C$ so that
$$A_k\le C(A(S))^pA_{k-1}(M_{k-1}+M_{k-2})^{p-1},$$
which, by \eqref{3.24}, implies that $A_k\le \tfrac12A_{k-1}$ for small
$\varepsilon'$.  Since $A_1$ is finite, the claim follows, which finishes
the proof of Theorem \ref{existthm}.
\qed

\end{document}